\documentclass[12pt]{article}

\usepackage{amsmath,amsthm,amssymb,amsfonts}
\usepackage{hyperref}
\usepackage{todonotes}

\usepackage[T1]{fontenc}\usepackage[upright]{fourier}\usepackage{baskervald}

\usepackage{dsfont}

\usepackage[all]{xy}

\usepackage{tikz-cd}

\usepackage[normalem]{ulem}

\usepackage{verbatim}

\theoremstyle{plain}
\makeatletter
\newtheorem*{rep@theorem}{\rep@title}
\newcommand{\newreptheorem}[2]{%
\newenvironment{rep#1}[1]{%
 \def\rep@title{#2 \ref{##1}}%
 \begin{rep@theorem}}%
 {\end{rep@theorem}}}
\makeatother

\newtheorem{thm}{Theorem}[section]
\newtheorem{lem}[thm]{Lemma}
\newtheorem{prop}[thm]{Proposition}
\newtheorem{cor}[thm]{Corollary}

\theoremstyle{definition}
\makeatletter
\newtheorem*{cont@example}{\cont@title}
\newcommand{\newcontexample}[2]{%
\newenvironment{cont#1}[1]{%
 \def\cont@title{#2 \ref{##1} continued}%
 \begin{cont@example}}%
 {\end{cont@example}}}
\makeatother

\newtheorem{defn}[thm]{Definition}

\newtheorem{rem}[thm]{Remark}

\theoremstyle{remark}

\newreptheorem{thm}{Theorem}
\newreptheorem{lem}{Lemma}
\newcontexample{exmp}{Example}
\newreptheorem{prop}{Proposition}

\begin{document}

\title{Dual fans and mirror symmetry}
\author{Patrick Clarke}
\date{\today}

\maketitle

\begin{abstract}
We show that the mirror constructions of Greene-Plesser, Berglund-H\"ubsch, Batryev-Borsov, Givental  and Hori-Vafa 
can be expressed in terms of what we call dual fans.  To do this, we associate to a pair of dual fans a pair of toric Landau-Ginzburg models,
and we describe a process by which each of the mirror constructions listed also produces a pair of toric Landau-Ginzburg models. Replacing 
mirror pairs by toric Landau-Ginzburg models is reversible, and our main result is the dual fan models and the mirror pairs models coincide.
\end{abstract}

\section{Introduction}

In this paper, we show that the mirror constructions of Greene-Plesser \cite{Greene_Plesser_1990}, Berglund-H\"ubsch \cite{Berglund-Hubsch}, Batyrev-Borisov \cite{Batyrev-Borisov-Cones}\footnote{As a consequence, 
Batyrev \cite{Batyrev} and Borisov \cite{Borisov} are also described this way.}, 
Givental \cite{Givental-1998}, and Hori-Vafa \cite{Hori-Vafa} can all be cleanly described in terms  of what
we call {\bf dual fans.}   These are fans $\Sigma \subset N,$ $\Sigma' \subset M$ of strongly convex rational polyhedral cones such that
\begin{itemize}
\item $M$ and $N$ dual finite rank free abelian groups, and 
\item $0 \leq \langle m, n \rangle $ for any $m \in |\Sigma'|$ and $n \in |\Sigma|.$
\end{itemize}

Associated to an ordered pair of dual fans $(\Sigma, \Sigma')$ is a {\bf toric Landau-Ginzburg model.}
This is a morphism
$$
W(\Sigma') \colon \mathbf{X}(\Sigma)\times {C(\Sigma')} \to  \mathbf{A}^1_{C(\Sigma')}
$$
over $C(\Sigma')= \operatorname{Spec}\mathbf{Z}[c_{\rho'}]_{\rho' \in \Sigma'(1)}$. Here $\mathbf{X}(\Sigma)$ is a toric ``variety'' (over $\mathbf{Z},$ or $\mathbf{C},$ or ...) 
formed in the usual way from $\Sigma,$ and 
$$
W(\Sigma') = \sum_{\rho' \in \Sigma'(1)} c_{\rho'} \chi^{u_{\rho'}}
$$
for  the character $\chi^{u_{\rho'}}$  on  $\mathbf{X}(\Sigma)$ corresponding to the integral generator $u_{\rho'} \in M$ of the ray $\rho'.$  Reversing the roles of the fans
$(\Sigma', \Sigma)$ produces the {\bf dual toric Landau-Ginzburg model}
$$
W(\Sigma) \colon \mathbf{X}(\Sigma')\times {C(\Sigma)} \to  \mathbf{A}^1_{C(\Sigma)}.
$$

Depending on the mirror construction, the mirror pair produced is either 
\begin{itemize}
\item a pair or toric Landau-Ginzburg models $W \colon X \to \mathbf{A}^1$ and $W' \colon X' \to \mathbf{A}^1$ \cite{Berglund-Hubsch},
\item complete intersections $Z,$ $Z'$  in a toric varieties $Y$ and $Y'$ \ \cite{Greene_Plesser_1990,Batyrev, Borisov, Batyrev-Borisov-Cones}, or 
\item a complete intersection $Z$ in a toric variety $Y$ and a toric Landau-Ginzburg model $W' \colon X' \to \mathbf{A}^1$ \cite{Givental-1998, Hori-Vafa}.
\end{itemize}
The complete intersections are presented as a global section of a split bundle over an ambient toric variety.  For instance, $Z \subseteq Y$ would be specified as the zero locus of some $g \in \Gamma(Y, \mathcal{V}),$ where 
 $\mathcal{V} = \bigoplus_{i=1}^c \mathcal{L}_i$ for invertible sheaves $\mathcal{L}_i.$

 In order to make our comparisons, to any member of a mirror pair we introduce an {\bf auxiliary toric Landau-Ginzburg model}: 
 $$
\mathbf{W} \colon \mathbf{X}(\Sigma_X)
 \times \Gamma \to \mathbf{A}^1_\Gamma
$$
where
\begin{itemize}
\item $X = 
\left\{ 
\begin{array}{ll}
X & \text{if the member is  $W \colon X \to \mathbf{A}^1,$} \\
\operatorname{\underline{Spec}} \operatorname{Sym^\bullet} \mathcal{V} & \text{if the member is specified by $g \in \Gamma(Y, \mathcal{V})$, }
\end{array}
\right.
$
\item
$
\Gamma = \operatorname{Spec} \mathbf{Z}[\gamma_m]_{m \in \Xi},
$ and
\item 
$
\mathbf{W}  = \sum_{m \in \Xi}  \ \gamma_m \chi^m,
$
where
\item 
$
\Xi = 
\left\{ 
\begin{array}{l}
\text{characters which appear  in the expansion $W = \sum_m g_m \chi^m,$} \\
\text{characters on $\operatorname{\underline{Spec}} \operatorname{Sym^\bullet} \mathcal{V}$ from $\Gamma(Y, \mathcal{V}).$ }
\end{array}
\right.
$
 \end{itemize}
 This toric Landau-Ginzburg model comes with equipped with a {\bf specialization}\footnote{Here the word ``specialization'' just means ``morphism,''  but using it allows us to refer specifically to these morphisms.} given by $\gamma_m \mapsto g_m$
 where the $g_m$'s come from either the expansion in characters of $W = \sum_m g_m \chi^m$  or  the section 
$g = \sum_m g_m \chi^m.$
 
It is somewhat surprising that, provided the specialization $\gamma_m \mapsto g_m$ is remembered, nothing is lost by replacing the complete intersection or Landau-Ginzburg model by 
  the auxiliary Landau-Ginzburg model.
   For instance, given
 $W \colon X \to \mathbf{A}^1$ it is easy to see that $W = \mathbf{W}_g.$  In the case of the complete intersection, observe that from $\mathbf{W}_g \colon X(\Sigma_X)_g \to \mathbf{A}^1$ we can 
recover $Y,$ $\mathcal{V},$ $g$ and thus $Z$ as
  \begin{itemize}
  \item $Y$ is the maximal proper toric stratum in $\mathbf{X}(\Sigma_X)_g,$  
  \item $\mathcal{V}$ is the conormal bundle, $\mathcal{I}_Y / \mathcal{I}_Y^2,$ of $Y$ in $\mathbf{X}(\Sigma_X)_g,$ and 
  \item $g = \mathbf{W}_g$ since $\mathbf{W}_g \in \mathcal{I}_Y.$ 
\end{itemize}
This way we treat any mirror pair in the same way: as a pair 
$$(\mathbf{W} \colon \mathbf{X}(\Sigma_X)\times \Gamma \to \mathbf{A}^1_\Gamma, \gamma \mapsto g) \text{ \quad  and \quad }
(\mathbf{W}' \colon \mathbf{X}(\Sigma_{X'})\times \Gamma' \to \mathbf{A}^1_{\Gamma'}, \gamma' \mapsto g').$$

Finally, our theorem is this:
\begin{thm}
\label{thm:main-theorem}
For any mirror pair,  $\Sigma_X$ and  $\Sigma_{X'}$ are dual fans,  
$\Gamma$ and $\Gamma'$ are affine spaces (by construction), and 
there are inclusions as coordinate subspaces
$$C(\Sigma_{X'}) \to \Gamma \text{ \quad and \quad} C(\Sigma_{X}) \to \Gamma'$$
such that the toric Landau-Ginzburg models given by the fans
$$
W(\Sigma_{X'}) \colon \mathbf{X}(\Sigma_X)\times {C(\Sigma_{X'})} \to  \mathbf{A}^1_{C(\Sigma_{X'})}
\text{ \quad and \quad} 
W(\Sigma_{X}) \colon \mathbf{X}(\Sigma_{X'})\times {C(\Sigma_{X})} \to  \mathbf{A}^1_{C(\Sigma_{X})}
$$
are obtained from
$$\mathbf{W} \colon \mathbf{X}(\Sigma_X)\times \Gamma \to \mathbf{A}^1_\Gamma \text{ \quad  and \quad }
\mathbf{W}' \colon \mathbf{X}(\Sigma_{X'})\times \Gamma' \to \mathbf{A}^1_{\Gamma'}$$
by base change.

Furthermore, if the mirror partner producing $\Sigma_X$ is 
a Landau-Ginzberg model (as opposed to a complete intersection) then then $C(\Sigma_X) \to \Gamma'$
is an isomorphism.
\end{thm}

\begin{rem}
Exhibiting $\Sigma_X$ and $\Sigma_{X'}$ as dual fans implies that the character group $M_X$
and the one-parameter subgroups $N_{X'}$ are identified.  In all the constructions here, except Greene-Plesser, this identification is built-in.  For the Greene-Plesser construction, there is a unique identification that makes the theorem true.
\end{rem}

\begin{rem}
There is no guarantee that the map given by the  specialization $\gamma \mapsto g$ factors through $C(\Sigma_{X'}).$   Nevertheless, these constructions are described by dual fans. 
If the mirror object is a Landau-Ginzburg model, then $C(\Sigma) = \Gamma'$
so the specialization factors.    In the case of a complete intersection, 
there is a process by which 
one can recover $\Gamma$ from $(\Sigma, \Sigma').$ Simply ask 
\begin{quote}
``Does $W(\Sigma') \colon \mathbf{X}(\Sigma)\times {C(\Sigma')} \to  \mathbf{A}^1_{C(\Sigma')}$ look like it came from $g \in \Gamma(Y, \mathcal{V}),$ for some $Y$, and $\mathcal{V}$?''
\end{quote}
When the answer is ``Yes,'' then we can recover 
recover $Y$  and $\mathcal{V}$ as before, and enlarge the base from $C(\Sigma')$ to $\Gamma.$
\end{rem}

We conclude the introduction some 
speculative remarks about  dual fan mirror symmetry.
\begin{rem}
In light of Theorem \ref{thm:main-theorem} an immediate question is ``To what extent do the computations normally compared in mirror symmetry match for the dual toric Landau-Ginzburg models formed from dual fans?''

Theorem \ref{thm:main-theorem}  provides a partial answer to this question in that in many cases both Hodge diamonds  \cite{Berglund-Hubsch, Greene_Plesser_1990, Batyrev-Borisov-Cones, krawitz:2010}
and Gromov-Witten invariants match \cite{Givental-1998}.

An additional question that arises when one considers the birational transformations that happen as one moves about within a toric variety's GKZ fan \cite{GKZ} is  ``To what extent do  mirror computations depend on $\Sigma(1)$ and $\Sigma'(1)$ alone?''
\end{rem}

\section{Toric geometry}
We review here some basic constructions in toric geometry, and give a description of the fan of a 
toric variety of the form $X = \operatorname{\underline{Spec}} \operatorname{Sym^\bullet} \mathcal{V}$
in Corollary \ref{cor:split-bundle-fan}.

\subsection{Fans}

\begin{defn}
A 
{\bf rational polyhedral cone}
$\sigma$ is a subset of a  $\mathbf{Q}$-vector space $V$  that 
can be written  
$$
\{ r_1 v_1 + \dotsm + r_s v_s \  | \ r_i \geq 0 \}
$$
for a finite set of vectors $v_1, \dotsc, v_s \in V.$  The vectors 
$v_1, \dotsc, v_s$ are said to {\bf generate} $\sigma.$
\end{defn}

\begin{defn}
A {\bf face} of a rational polyhedral cone $\sigma$ in $V$ is a subset of the form 
$$
\{ v \in \sigma \ | \ \ell(v) = 0 \}
$$
for a linear functional $\ell \colon V \to \mathbf{Q}$ with the property $
\{ v \in \sigma \ | \ \ell(v) < 0 \}
$ is empty.  A maximal proper face is called a {\bf facet}.
\end{defn}

\begin{defn}
The {\bf dimension} of a cone $\sigma$ is the dimension of the vector space $\operatorname{span}_\mathbf{Q} \sigma.$
\end{defn}

\begin{defn}
A cone $\sigma$ is called {\bf strongly convex} if $0 \in V$ is a face of $\sigma.$
\end{defn}

\begin{defn}
A $1$-dimensional strongly convex cone $\rho \subset N$ is called {\bf ray}.  
The monoid $\rho \cap N$ is isomorphic to $\mathbf{N},$ and we denote its {\bf generator}
by $u_\rho.$
\end{defn}

\begin{defn}
Given a finitely generated free abelian group $N$, a {\bf fan}  $\Sigma$ in $N$ is a collection  of strongly convex rational cones in $N_\mathbf{Q} = N \otimes_\mathbf{Z} \mathbf{Q}$ such that 
\begin{itemize}
\item The face of any cone is a member of $\Sigma,$ and
\item The intersection of any two cones in $\Sigma$ is a face of each.
\end{itemize}
\end{defn}

\begin{defn}
The set of $k$-dimensional cones in a fan $\Sigma$ is denoted $\Sigma(k).$
\end{defn}

\subsection{The toric variety of a fan}

We recall the  standard procedure \cite{demazure-cremona} by which one produces a scheme from a fan.

\begin{defn}
Given a cone $\sigma$ in $V,$ the {\bf dual cone} $\sigma^\vee$
is the subset of $\operatorname{Hom}_\mathbf{Q}(V, \mathbf{Q})$
made up of elements
$$
\{ \ \ell \  | \ \ell(v) \geq 0,  \forall v \in \sigma\}.
$$
\end{defn}

\begin{defn}
If $(\tau, +)$ is a commutative monoid, we denote the {\bf monoid algebra} by $\mathbf{Z}[\tau].$
For an element of $\tau,$ we often write  the associated {\bf monomial} with the element as the exponent of a {\bf base variable}. For instance,
if $v, v' \in \tau,$ and  $x$ as the base variable we have
$$
x^v \cdot x^{v'} = x^{v+v'}.
$$
\end{defn}

\begin{defn}
Denote $M = \operatorname{Hom}_\mathbf{Z}(N, \mathbf{Z}).$
Associated to a fan $\Sigma \subset N$ there is an integral Noetherian scheme  ${\mathbf{X}({\Sigma})}$
with rational functions
$$
\mathbf{Q}[M]
$$
and an open cover
$$
U_\sigma = \operatorname{Spec} \mathbf{Z}[\sigma^\vee \cap M]
$$
for $\sigma \in \Sigma.$  This is called the {\bf toric scheme}\footnote{These properties determine ${\mathbf{X}(\Sigma)}.$} of $\Sigma.$
\end{defn}

\begin{rem}
Setting $
U_\sigma = \operatorname{Spec} \mathbf{C}[\sigma^\vee \cap M]
$ in the above definition yields the usual construction of a toric variety over $\mathbf{C}.$  This can also be achieved by changing base from $\mathbf{Z}$ to $\mathbf{C}.$
\end{rem}

\subsubsection{Quotient fans}
\label{subsection:quotient-fans}
It is often convenient to describe a fan $\Sigma \subset N$ as  the image under a homomorphism $Q \colon \tilde{N} \to N$
of a fan  $\tilde{\Sigma} \subset  \tilde{N}.$
This makes sense provided 
$$
\Sigma = \{ P(\sigma) \subset N \ | \ \sigma \in \tilde{\Sigma} \}
$$
is a fan.  
Under this assumption, the resulting toric scheme $\mathbf{X}(\Sigma)$ is the product 
$$ \mathbf{X}({\tilde{\Sigma}}) / G  \ \times  \  \operatorname{Spec} \mathbf{Z}[K],$$
where 
$K$ is the kernel and $C$ is the cokernel 
of the map $Q^t \colon M \to \tilde{M},$
and
$G = \operatorname{Spec} \mathbf{Z}[C].$ 

This follows easily from the fact that the assignment $\Lambda \mapsto \operatorname{Spec} \mathbf{Z}[\Lambda]$
 is a version of {\bf Cartier duality}.
If one defines the Cartier dual  of $\mathbf{Z}$ to be $\mathbf{G}_m,$
then
this fits in perfectly with the fact that the big torus a toric scheme is the spectrum of  the  group ring of its characters. 
When the abelian group is finite, the Cartier dual equals the Pontryagin dual.

\subsection{Normal fans}

\begin{defn}
A {\bf  rational convex polyhedral set} $P \subseteq M_\mathbf{R}$ is any subset of the form 
$$
P = \{ m \in M_\mathbf{R} \ | \ A(m) + \ell \geq 0 \}
$$
for a homomorphism $A \colon  M \to \mathbf{Z}^{n}$ and an element $\ell \in  \mathbf{R}^{n}.$
The dimension of $P$ is the dimension of the vector space $\operatorname{span}_\mathbf{R} P.$
If $\dim P = \dim M_\mathbf{R},$ then $P$ is called {\bf full dimensional}.
\end{defn}

\begin{defn}
A full dimensional rational convex polyhedral set   $
P
$ defines an {\bf normal fan} $\Sigma_P \subset N.$ The cones of $\Sigma_P $ are labeled by faces $F$ of $P:$
$$
\sigma_F = \{ n \in N_{\mathbf{Q}} \ | \ n(p-p_F) \geq 0 \text{ for all } p \in P \text{ and } p_F \in F\}.
$$
\end{defn}

\subsection{Completely split bundles over a toric variety.}
\label{subsection:split-bundles}

\begin{defn}
Given a fan in $\Sigma \subset N,$ to each ray $\rho \in \Sigma(1)$
we have a {\bf torus invariant Weil divisor} $D_\rho$ whose valuation on characters
$$
M \to \mathbf{Z}. 
$$
is given by $m \mapsto \langle m , u_\rho \rangle.$
\end{defn}

\begin{thm}\cite[Theorem 4.2.8 (a) and (c)]{Cox-Little-Schenck}
\label{theorem:linear-equals-cartier}
An element $D = \sum_{\rho \in  \Sigma(1)} a_\rho D_\rho \in  \mathbf{Z}^{\Sigma(1)}$ 
defines a Cartier divisor on $Y$ if and only if for each cone $\sigma \in \Sigma$ there is a homomorphism
$a_\sigma \colon \sigma \cap N \to \mathbf{Z}$ such that $a_\sigma(u_\rho) = a_\rho$ for all $1$-cones $\rho \subseteq \sigma.$  
\end{thm}

\begin{prop} {\bf (line bundle fan) }
\label{proposition:line-bundle-fan}
Let $D = \sum_{\rho \in  \Sigma(1)} a_\rho D_\rho \in  \mathbf{Z}^{\Sigma(1)}$ be a Cartier divisor
on the toric variety $Y$.
Set  $\mathcal{V}  =  \mathcal{O}(D),$ and 
$$X = \operatorname{\underline{Spec}} \operatorname{Sym^\bullet} \mathcal{V}.$$
Then  $X$ is toric with fan 
$\Sigma_X \subset N \oplus \mathbf{Z}.$ The cones of
$\Sigma_X$  can be described in terms of those of $\Sigma_Y$: For each $\sigma \in \Sigma,$
we have the {\bf lifted cone}
$$
\hat{\sigma}_0 = \text{ the cone generated by } \{ (u_\rho, a_\rho) \ | \ \rho \in \sigma(1) \}
$$
and the {\bf over cone}
$$
\hat{\sigma} = 
\hat{\sigma}_0 +  \hat{\rho}_{\text{vert}} 
$$
where the second summand is the {\bf vertical ray}
$$
\hat{\rho}_{\text{vert}} =   \mathbf{Q}_{\geq 0} \cdot (0,1).
$$
\end{prop}
\begin{proof}
Choose a rational section $p$ of $\mathcal{O}(D)$ whose divisor is $D.$  Since $Y$ 
is integral so is $X,$ and we can identify the rational functions on $X$ with $\mathbf{Q}[M \oplus \mathbf{Z}],$ where the $\mathbf{Z}$ summand corresponds to powers of $p.$

Now consider $X$ over $U_\sigma \subseteq Y$ for some $\sigma \in \Sigma.$
The regular functions on the affine scheme $X \times_Y U_\sigma$ are given
by a subring of  $\mathbf{Q}[M \oplus \mathbf{Z}].$  Indeed, we may immediately restrict
to $\mathbf{Z}[M \oplus \mathbf{Z}].$  The divisor of the monomial $y^m p^\ell$ 
on  $X \times_Y U_\sigma$ is
$$
\ell \ (Y \cap U_\sigma) +  \sum_{\rho \in \sigma(1)} (u_\rho(m) - \ell a_\rho) \ \ L \times_{Y} (D_\rho \cap U_\sigma) 
$$
where we have written $Y \cap U_\sigma$ for the copy of $U_\sigma$ in the zero section.
Thus it is regular 
on $X \times_Y U_\sigma$ if and only if $\ell \geq 0 $ and $u_\rho(m) - \ell a_\rho \geq 0$ for all $\rho \in \sigma(1).$
These conditions cut out a monoid in $M \oplus \mathbf{Z},$ and thus a subring of 
$\mathbf{Z}[M \oplus \mathbf{Z}].$ 

The elements $\hat{\sigma}$ in $(N \oplus \mathbf{Z})_{\mathbf{Q}}$ which evaluate positively on this monoid are exactly the elements in the cone generated by  $(0,1)$  and $ (u_\rho, a_\rho).$  This is the cone $\hat{\sigma}.$  The cone $\hat{\sigma}_0$ is a face of this cone, and it 
is straightforward to check that these cones form a fan, and the corresponding open sets 
$U_{\hat{\sigma}} = X \times_Y U_{\sigma}$ cover $X$.
\end{proof}

\begin{cor} {\bf (split bundle fan)}
\label{cor:split-bundle-fan}
Iterating this procedure, one can describe 
$$X = \operatorname{\underline{Spec}} \operatorname{Sym^\bullet} \mathcal{V}$$
for $\mathcal{V} = \bigoplus_{i=1}^c \mathcal{L}_i$ where 
$\mathcal{L}_i = \mathcal{O}(D^{(i)})$
for a Cartier divisor
$D^{(i)} = \sum_{\rho \in  \Sigma(1)} a^{(i)}_\rho D_\rho$
by a fan
$$
\Sigma_X \subset N \oplus \mathbf{Z}^c
$$
whose cones can be written in term of those of $\Sigma_Y$ by setting
$$
\hat{\sigma}_{(0, \dotsc,0)} = \text{ the cone generated by }  \{ (u_\rho,   a^{(1)}_\rho, \dotsc,  a^{(c)}_\rho) \ | \ \rho \in \sigma(1)   \}
$$
and then 
$$
\hat{\sigma}_{(b_1, \dotsc, b_c)}  = \hat{\sigma}_{(0, \dotsc,0)} + \sum_i \mathbf{Q}_{\geq 0} \cdot  b_i (0, e^\vee_i)
$$
where $b_i \in \{0,1\}$  and $e^\vee_i$ is the $i^\text{th}$ standard basis vector of $\mathbf{Q}^c.$ 
\end{cor}

\subsection{Global functions on toric varieties.}

\begin{prop} {\bf (regularity)}
\label{prop:regularity}
A rational function   $W = \sum_{\chi} c_\chi \chi$  expanded in characters on a toric variety $X$ is an element of $\mathcal{O}_X(X)$ if and only if for each $\chi$ with non-zero coefficient in the expansion  and all  $\rho \in \Sigma_X(1)$
$$\langle \chi, u_\rho \rangle \geq 0.$$
\end{prop}

\begin{cor} {\bf (dual fans and regularity)}
The following are equivalent:
\begin{itemize}
\item $(\Sigma,\Sigma')$ are dual fans.
\item $W(\Sigma')$ is a regular function on $\mathbf{X}(\Sigma)\times {C(\Sigma')}.$
\item $W(\Sigma)$ is a regular function on $\mathbf{X}(\Sigma')\times {C(\Sigma)}.$
\end{itemize}
\end{cor}

\section{Comparison to existing constructions}

We now make our way through the mirror constructions, verifying Theorem 
\ref{thm:main-theorem}.  In each case the main steps we take are to
\begin{itemize}
\item recall the {\bf original mirror construction},
\item identify the {\bf auxiliary data} 
($\Sigma_X, \Gamma,  \mathbf{W}, \gamma \mapsto g$, 
and 
$\Sigma_{X'}, \Gamma', \mathbf{W}', \gamma' \mapsto g'$),
\item verify $\Sigma_X$ and $\Sigma_{X'}$ are {\bf dual fans},
\item define {\bf inclusions} $C(\Sigma') \to \Gamma$ and $C(\Sigma) \to \Gamma',$ and finally
\item we verify that the LGs given by the dual fans $\Sigma_X$ and $\Sigma_{X'}$
are  produced by {\bf base change} to 
$C(\Sigma')$ and $C(\Sigma)$
from the auxiliary ones.
\end{itemize}

\subsection{The quintic threefold \cite{Candelas-delaOssa-Green-Parks}}

Here both sides of the mirror are hypersurfaces, so $\Gamma$ and $\Gamma'$ are the affine spaces of global sections, and $\mathbf{W}$ and $\mathbf{W}'$ are the universal sections considered as functions.

\begin{defn} {\bf (Candelas-de la Ossa-Green-Parks family of quintics and its mirror)}
 Candelas-de la Ossa-Green-Parks \cite{Candelas-delaOssa-Green-Parks}  considers the family  depending on the parameter $\psi$
of subschemes of $\mathbf{P}^4$ 
defined by the  polynomial 
\begin{equation}
\label{equation:quintic}
g  = Y_0^5 + Y_1^5 + Y_2^5 + Y_3^5 + Y_4^5 - 5 \psi Y_0 Y_1 Y_2 Y_3 Y_4.
\end{equation}
The mirror family is formed by taking a $(\mathbf{Z}/5\mathbf{Z})^3$ quotient of these quintics.  
Explicitly, the group action is
\begin{equation}
\label{equation:z5z-action}
 (\zeta_1, \zeta_2, \zeta_3) \star (Y_0, Y_1, Y_2, Y_3)  = (   Y_0, \  \zeta_1 Y_1,  \  \zeta_2 Y_2,  \ \zeta_3  Y_3, \  (\zeta_1 \zeta_2 \zeta_3)^{4} Y_4 )
 \end{equation}
 where the $\zeta$'s are $5^\text{th}$ roots of unity.
\end{defn}

\begin{defn} {\bf (base fan)}
Writing $F(-)$ for the free group, 
the fan of $\mathbf{P}^4$ is 
\begin{equation*}
\Sigma_Y \subset
N = F(\{ u_{\rho_0}, u_{\rho_1}, u_{\rho_2}, u_{\rho_3}, u_{\rho_4} \}) / \langle 
u_{\rho_0} + u_{\rho_1}+ u_{\rho_2} + u_{\rho_3} + u_{\rho_4}
\rangle
\end{equation*}
with cones given by the $\mathbf{Q}_{\geq 0}$-spans of the subsets of size $\leq 4$ 
of the generators.
\end{defn}

\begin{prop} {\bf ($\Sigma_X$)}
With $\mathcal{V} = \mathcal{O}(5),$
the fan of the toric variety $X =\operatorname{\underline{Spec}} \operatorname{Sym^\bullet} \mathcal{V}$
is 
\begin{equation*}
\Sigma_X \subset \overline{N} = N \oplus \mathbf{Z}
\end{equation*}
with cones given by the $\mathbf{Q}_{\geq 0}$-spans of the subsets of 
$$
\{(0,1), (u_{\rho_0}, 1), (u_{\rho_1}, 1), (u_{\rho_2}, 1) , (u_{\rho_3},1), (u_{\rho_4},1) \}
$$ 
that exclude at least one element of the form $(u_{\rho_i},1).$
\end{prop}
\begin{proof}
Following the method
of Proposition \ref{proposition:line-bundle-fan},
 we choose the divisor $$D=D_0+D_1+D_2+D_3+D_4$$ where $D_i = \{[Y_0 : \dotsm : Y_4] \ | \ Y_i=0\}.$  Then by excluding  at least one of the lifted rays in forming our cones, we get all the cones described in the proposition.
\end{proof}

\begin{defn}{\bf (invariant characters)}
The characters $\overline{M}$ on $X$ are naturally thought of as those Laurent monomials in $Y_0, \dotsc, Y_4$  with degree divisible by $5.$
Write $\overline{M}' \subseteq \overline{M}$ for those monomials invariant under the action of $(\mathbf{Z}/5\mathbf{Z})^3$
given in Equation (\ref{equation:z5z-action}), and set $\overline{N}' = \operatorname{Hom}_\mathbf{Z}(\overline{M}', \mathbf{Z}).$
\end{defn}

\begin{prop}{\bf ($\Sigma_{X'}$)}
The fan $
\Sigma_{X'} \subset \overline{N}' 
$ of $X' = X/(\mathbf{Z}/5\mathbf{Z})^3$ 
is the quotient fan of $\Sigma_X$ under the map $\overline{N} \to \overline{N}'$
transpose to the inclusion $\overline{M}' \hookrightarrow \overline{M}.$
\end{prop}
\begin{proof}
This is an example of a situation in which the quotient is presented by the quotient fan as described in Subsection \ref{subsection:quotient-fans}. 
\end{proof}

\begin{prop} {\bf ($\overline{N}' = \overline{M}$)}
The assignment 
$$
(n, k) \mapsto (Y_0 Y_1Y_2 Y_3 Y_4)^k
\prod_i (Y_i^5/Y_0 Y_1 Y_2 Y_3 Y_4)^{n_i} 
$$
for $n = n_{0} u_{\rho_0} + \cdots + n_{5} u_{\rho_5} \in N$
and $k \in \mathbf{Z}$
injectively maps $\overline{N}$ into $\overline{M}$
 and with image $\overline{M}'.$  The transpose isomorphism $\overline{N}' \to \overline{M}$ allows us to meaningfully write
 $$
 \Sigma_{X'} \subset \overline{M}.
 $$
\end{prop}
\begin{proof}
To be sure, $M'$ is made up of monomials with degree $0,$ and is  freely generated by  those 
 with exponent vectors 
\begin{equation}
\label{equation:exponent-vectors}
\{ (-1,  \ 4, -1, -1, -1), 
( -1, -1, \ 4, -1,  -1), 
(-1, -1, -1, \ 4,  -1), 
(-1, -1, -1, -1, \ 4)
\}.
\end{equation}
A computational check shows  the assignment 
 \begin{equation} 
 \label{equation:sigma'-in-M} 
 u_{\rho_i} \mapsto Y_i^5/ Y_0 Y_1 Y_2 Y_3 Y_4
 \end{equation}
 gives a well defined homomorphism and 
 identifies $N$ and $M'.$  Finally, sending $(0,1) \mapsto Y_0 Y_1 Y_2 Y_3 Y_4$ and $$\overline{M}' =  
\bigcup_k  \ (Y_0 Y_1 Y_2 Y_3 Y_4)^k \ 
M'$$
completes the isomorphism $\overline{M}' = \overline{N}.$
\end{proof}

\begin{prop} {\bf ($\Gamma, \mathbf{W}, \gamma \mapsto g$)}
\label{prop:quintic-associated-LG}
The global sections $\Gamma(\mathbf{P}^4, \mathcal{V})$ considered as functions on $X$ 
has a basis of monomials
$\Xi = \{ \text{degree 5 monomials in $Y_0, \cdots, Y_4 $} \}.$
Thus, 
$$
\Gamma = \operatorname{Spec}  \  \mathbf{Z}[\gamma_m]_{m \in \Xi}
$$
and the auxiliary Landau-Ginzburg model of the quintic is 
$$
\mathbf{W} \colon {\bf X}(\Sigma_X)_\Gamma \to \mathbf{A}^1_\Gamma
$$
where 
$$
\mathbf{W} = \sum_{m \in \Xi} \gamma_m Y^m.
$$
The specialization
$$
\mathbf{A}^1_{\psi} \to \Gamma
\quad \text{ given by sending } \quad 
\gamma_{m} \mapsto 
\left\{
\begin{array}{cl}
-5\psi & \ \ \text{if $m = Y_0 Y_1 Y_2 Y_3 Y_4,$} \\
1 &  \ \ \text{if $m = Y_i^5$ for some $i$,} \\
0 &  \ \ \text{otherwise} \\
\end{array}
\right.
$$
sets the coefficients to those used in (\ref{equation:quintic}).
\end{prop}
\begin{proof}
These statements are immediate.
\end{proof}

\begin{prop} {\bf ($\Gamma', \mathbf{W'}, \gamma' \mapsto g'$)}
\label{prop:quintic-quot-associated-LG}
Functions on $X'$ which pullback to elements of $\Gamma(Y, \mathcal{V})$ 
have a basis of monomials
$$\Xi' = \{ \text{degree 5 monomials in $Y_0, \cdots, Y_4 $ invariant under the action 
$(\mathbf{Z}/5\mathbf{Z})^3$}\}.$$ Thus  
$$
\Gamma' = \operatorname{Spec}\  \mathbf{Z}[\gamma'_m]_{m \in \Xi'}
$$
and the 
auxiliary Landau-Ginzburg model of the quintic-mirror is 
$$
\mathbf{W'} \colon X(\Sigma')_{\Gamma'} \to \mathbf{A}^1_{\Gamma'}
$$
where 
$$
\mathbf{W'}  = \sum_{m \in \Xi'} \gamma'_m Y^m.
$$
The specialization
$$
\mathbf{A}^1_{\psi} \to \Gamma'
\quad \text{ given by sending } \quad 
\gamma'_{m} \mapsto 
\left\{
\begin{array}{cl}
-5\psi & \ \ \text{if $m = Y_0 Y_1 Y_2 Y_3 Y_4,$} \\
1 &  \ \ \text{if $m = Y_i^5$ for some $i$,} \\
0 &  \ \ \text{otherwise} \\
\end{array}
\right.
$$
sets the coefficients to those used in (\ref{equation:quintic}).
\end{prop}
\begin{proof}
These statements are immediate.
\end{proof}

\begin{thm} {\bf (quintic main)}
$\Sigma_X$ and $\Sigma_{X'}$ are dual fans, 
\begin{itemize}
\item $C(\Sigma_X), C(\Sigma_{X'})$ and $\Gamma'$ are isomorphic,
\item there is a closed immersion $\Gamma' \hookrightarrow \Gamma,$ and 
\item a closed immersion $\mathbf{A}^1_\psi \hookrightarrow C(\Sigma_X),$
\end{itemize}
such that 
the dual toric Landau-Ginzburg models of $\Sigma_X$ and $\Sigma_{X'}$ are obtained from 
the associated toric Landau-Ginzburg models of Proposition \ref{prop:quintic-associated-LG} via base change, and the specializations of
Proposition \ref{prop:quintic-associated-LG} factor though the map
 $\mathbf{A}^1_\psi \hookrightarrow C(\Sigma_X).$
\end{thm}

\begin{proof}
The map  in Equation (\ref{equation:sigma'-in-M}) produces the map $C(\Sigma_X) \hookrightarrow \Gamma'$
by sending 
\begin{itemize}
\item $c_{(u_{\rho_i}, 1)} \mapsto \gamma'_{Y_i^5}$ and
\item $c_{(0,1)} \mapsto \gamma'_{Y_0 Y_1 Y_2 Y_3 Y_4}.$
\end{itemize}
Most of these statements follow from explicitly writing the maps involved.
The closed immersion 
$\Gamma' \hookrightarrow \Gamma$  sends $\gamma_m \mapsto \gamma'_m$ if $m$ is invariant under the $(\mathbf{Z}/5\mathbf{Z})^3$ action, and  $0$ otherwise.
The isomorphism $C(\Sigma_{X'}) \to C(\Sigma_X)$ comes from the fact that the $1$-cones are in bijection under the quotient map $\Sigma_X \to \Sigma_{X'}$.

The expression
$$
c_{(\rho_0, 1)}  Y_0^5 +
c_{(\rho_1, 1)}  Y_1^5 +
c_{(\rho_2, 1)}  Y_2^5 +
c_{(\rho_3, 1)}  Y_3^5 +
c_{(\rho_4, 1)}  Y_4^5 +
c_{(0,1)} Y_0 Y_1 Y_2 Y_3 Y_4
$$
for $W(\Sigma_X)$
produces most of the rest.
In light of the identification $C(\Sigma_{X'}) \to C(\Sigma_X)$ this is also an expression for 
$W(\Sigma'),$ and this
immediately yields the base-change statements.  The regularity of these functions guarantees the fans are dual by the observation made in Proposition \ref{prop:regularity}.

The remaining question whether $C(\Sigma_X) \hookrightarrow \Gamma'$
is an isomorphism
is settled by noticing the matrix with rows from Equation 
(\ref{equation:exponent-vectors}) is the map $M' \to M$ in the exact sequence
$$
0 \to M' \to M \to (\mathbf{Z}/5\mathbf{Z})^3 \to 0.
$$
\end{proof}

\subsection{Berglund-H\"ubsch-Krawitz \cite{Berglund-Hubsch}  \cite{krawitz:2010}. }

In this case, both sides of the mirror are Landau-Ginzburg models so $\mathbf{W}$ and $\mathbf{W}'$ are 
formed by simply inserting variables for the coefficients of $W$ and $W'.$  $\Gamma$ and $\Gamma'$ 
are the affine spaces with these coefficient variables as coordinates.

\begin{defn} {\bf (polynomials and phase symmetries of a  matrix)}
The duality of Berglund-H\"ubsch \cite{Berglund-Hubsch}
is  based on invertible  matrices with nonnegative integer entries.
Given such an $(n+1) \times (n+1)$-matrix $P = (p_{ij})_{ij},$ 
one can 
define the surjective group homomorphism 
$$
\mathcal{F}(P) \colon (\mathbf{C}^\times)^{n+1} \to (\mathbf{C}^\times)^{n+1}
$$
given by the assignment 
$y_j \mapsto x_0^{p_{0j}} \dotsm x_n^{p_{nj}}.$  Abstractly,  $\mathcal{F}(-) = \operatorname{Spec} \mathbf{C}[-].$
From this we have a polynomial 
$$
W_P =  (\sum_{j=0}^n y_j) \circ \mathcal{F}(P) =  \sum_{j =0}^n \prod_{i=0}^n x_i^{p_{ij}},
$$
and a group of {\bf phase symmetries}
$$
S_P = \ker \mathcal{F}(P).
$$

The transpose matrix $P^t$ defines what is referred to as the {\bf transpose polynomial} $W_{P^t}$
and the {\bf transpose phase symmetries} $S_{P^t}.$
\end{defn}

\begin{defn} {\bf (quantum symmetries and BH mirror criterion)}
The Berglund-H\"ubsch \cite{Berglund-Hubsch} {\bf mirror criterion} is formulated in terms of 
choices of groups of {\bf quantum symmetries} $Q_P \leq S_P$ and $Q_{P^{t}} \leq S_{P^t}.$
These can be any subgroups and constitute a dual pair if the cokernels $G_P$ in 
$$
1 \to Q_P \to S_P \to G_P \to 1
$$ 
and  $G_{P^t}$ in 
$$
1 \to Q_{P^t} \to S_{P^t} \to G_{P^t} \to 1
$$
satisfy
$Q_P \cong G_{P^t}$
and 
$Q_{P^t} \cong G_P.$
Typically such a pair is presented as 
$$(W_P, Q_P), (W_{P^t}, Q_{P^t}).$$
\end{defn}

\begin{defn} {\bf (Krawitz's dual group)}
Krawitz  \cite{krawitz:2010}  discovered a way to systematically 
produce a group $Q_{P}^t \leq S_{P^t}$ dual to any given group of quantum symmetries $Q_P.$
A reformulation {\bf Krawitz's dual} by Clarke \cite{clarke-2013} first considers the diagram of character groups
\begin{equation}
\label{equation:bh-diagram}
\begin{tikzcd}
\mathbf{Z}^{n+1}  & M \arrow[swap]{l}{A} \\
\mathbf{Z}^{n+1} . \arrow{u}{P}  \arrow[swap]{ur}{B^t} 
\end{tikzcd}
\end{equation}
associated to the homomorphisms
$$
\begin{tikzcd}
(\mathbf{C}^\times)^{n+1}  \arrow{r}{\mathcal{F}(A)} \arrow[swap]{d}{\mathcal{F}(P)} & (\mathbf{C}^\times)^{n+1}/ Q_{P} \arrow{dl}{\mathcal{F}(B^t) } \\
(\mathbf{C}^\times)^{n+1} & .
\end{tikzcd}
$$
$Q_{P}^t$ is then defined to be the 
 the kernel of 
 $$\mathcal{F}(B) \colon (\mathbf{C}^{\times})^{n+1} \to U'$$
$U'= \operatorname{Spec} \mathbf{C}[N]$ for $N =  \operatorname{Hom}_{\mathbf{Z}}(M, \mathbf{Z}).$
\end{defn}

\begin{defn} {\bf (BHK mirror LG models)}
Berglund-Hubsch-Krawitz {\bf mirror pairs} arise as Landau-Ginzburg models associated to $(P, Q_P)$
and $(P^t, Q^t_P).$ 
The Landau-Ginzburg model associated  $P$ and a group of quantum symmetries $Q_P,$ 
is
$$W_P\colon \mathbf{C}^{n+1}/Q_P \to \mathbf{C}.$$ 
In the same way $P^t$ and $Q_{P}^t$ defines a Landau-Ginzburg model.  Together, these two constitute a BHK mirror pair.
\end{defn}

\begin{prop} {\bf ($\Sigma, \Sigma'$)}
The fan $\Sigma$ of $\mathbf{C}^{n+1}/Q_{P}$ is the image of the positive orthant fan 
under the map $A^t \colon \mathbf{Z}^{n+1} \to N,$ and $W_P$ is given by the image $\Sigma'$ of 
the positive orthant fan under the map $B \colon \mathbf{Z}^{n+1} \to M$ of Equation (\ref{equation:bh-diagram}).
\end{prop}
\begin{proof}
This is an example of the situation in which the geometric quotient is presented by the quotient fan as described in Subsection \ref{subsection:quotient-fans}. 
\end{proof}

\begin{defn} {\bf ($\mathbf{W}, \mathbf{W}', \Gamma, \Gamma'$ and $\gamma \mapsto g, \gamma' \mapsto g' $)}  For the auxiliary Landau-Ginzburg models we have
\begin{itemize}
\item $\mathbf{W}  = \sum_{j =0}^n \gamma_j \prod_{i=0}^n x_i^{p_{ij}},$ and
\item $\mathbf{W}'  = \sum_{i =0}^n \gamma'_i \prod_{j=0}^n {x'}_j^{p_{ij}}$
\end{itemize}
where 
\begin{itemize}
\item $\Gamma = \operatorname{Spec} \mathbf{Z}[\gamma_j]_j,$ and
\item $\Gamma' = \operatorname{Spec} \mathbf{Z}[\gamma'_i]_i.$
\end{itemize}
The specializations to recover the BHK mirrors set all the $\gamma$ and $\gamma'$ variables to $1$.
\end{defn}

\begin{thm} {\bf (BHK main)}
The fans $(\Sigma, \Sigma')$ are dual with 
$$
p_{ij} = u_{\rho'_j}(u_{\rho_i}),
$$ 
there are isomorphisms
$C(\Sigma') = \Gamma,$
$C(\Sigma) = \Gamma',$
$W(\Sigma') = \mathbf{W},$ $W(\Sigma) = \mathbf{W}'$
and $W_P$ and $W_{P^t}$ are obtained from these via base change.
\end{thm}
\begin{proof}
These statements are immediate from the definition and considerations above.
\end{proof}

\newpage
\subsection{Batyrev-Borisov  \cite{Batyrev-Borisov-Cones}}

Some work is required to explicitly extract the mirror complete intersections in this construction.   $\Gamma$ and $\Gamma'$ are the affine spaces of global sections, and $\mathbf{W}$ and $\mathbf{W}'$ are the universal sections considered as functions.

\subsubsection{Gorenstein cones and splittings.}

Batyrev-Borisov mirror symmetry  \cite{Batyrev-Borisov-Cones} is centered on 
 {\bf complete splittings} of {\bf reflexive Gorenstein cones}.
 
\begin{defn} 
A full dimensional cone $K$ in a the reification $\overline{M}_\mathbf{R}$ of a finite rank free abelian group $\overline{M}$
is called {\bf Gorenstein} if there is a homomorphism $\ell^\vee \colon K \cap \overline{M} \to \mathbf{N}$
such that $K\cap \overline{M}$ is generated as a monoid by $(\ell^\vee)^{-1}({\{1\}}).$
\end{defn}

\begin{defn}
A Gorenstein cone $K$ is called {\bf reflexive} if the dual cone $K^\vee$
is also Gorenstein.  The number $\ell^\vee (\ell) = r$ is called the {\bf index} of $K$.
\end{defn}
\begin{defn} 
A {\bf complete splitting} of an index $r$ reflexive Gorenstein cone $K \subseteq \overline{M}$
is a set of non-zero elements $E = \{ e_1, \dotsc, e_r \}  \subseteq K \cap \overline{M}$ such that 
$$
\ell = e_1 + \cdots + e_r.
$$
\end{defn}

\subsubsection{Dual splittings, support  and partitions}

If a Gorenstein cone has a  splitting, then  so does its dual.  The proof of this uses the notions of 
both the {\bf support} of a Gorenstein cone and  the {\bf support partition} given by a complete splitting.

\begin{defn}
A {\bf  dual complete splitting} $E^\vee$ to a given complete splitting $E$ is a complete splitting  
$E^\vee = \{
e^\vee_1, \dotsc, e^\vee_r
\}$ of $K^\vee.$ 
\end{defn}

\begin{rem} {\bf (non-uniqueness of splittings)}
A splitting  need not be unique.  Indeed, 
Batyrev-Nill  give a simple explicit example where it is not \cite[Example 5.1]{batyrev-nill}. 
However, if $r=1$ it's clear that there is only one ``splitting,''  $e_1 = \ell,$
so uniqueness is guaranteed.  
\end{rem}

\begin{defn} 
Associated to a Gorenstein cone $K$ is a convex polyhedral set 
$$
\tilde{\Delta} = \{ k \in K \ | \ \ell^\vee(k) = 1 \in \mathbf{R} \}.
$$
called the {\bf support} of $K.$  We will denote the support of $K^\vee$
by $\tilde{\nabla}.$  
\end{defn}
\begin{defn} {\bf (support partition)}
 A complete splitting $E^\vee$ determines sets
$$\tilde{\Delta}_i = \{ k \in K \ | e^\vee_i(k) =1  \text{ and } e^\vee_j(k) =0
\text{ for all }  j \neq i \} \subseteq \tilde{\Delta}.$$
\end{defn}
\begin{prop} {\bf (support partitions partition)}
\label{proposition:part-part}
The sets $\tilde{\Delta}_i$ partition the integral elements of $\tilde{\Delta}$, 
each  $\tilde{\Delta}_i$ is nonempty and is the convex hull of vertices of $\tilde{\Delta}$
it contains.
\end{prop}
\begin{proof}
The partition assertion follows from the fact that any integral point $p$ in $\tilde{\Delta}$ has 
$\ell^\vee(p) = \sum_i e^\vee(p) = 1$ and $e_i^\vee(p) \in \mathbf{N}.$  This means that
$e_i^\vee(p) = 1$ for exactly one $i$ and zero for all others, i.e. $p \in \tilde{\Delta}_i.$

Nonemptyness can be seen by considering $\ell \in K$ and the fact that
$
\ell = p_1 + \cdots + p_r
$  
for not necessarily distinct $p_j \in \tilde{\Delta}.$  If $\tilde{\Delta}_i = \varnothing$
then $e_i^\vee(p_j) = 0$ for all $j,$ and consequently $\ell(e_i^\vee) = 0.$  This is a contradiction.
In fact, we have shown that each $\tilde{\Delta}_i$ contains  exactly one $p_j$ in any expression
of $\ell$ as a sum of element from $\tilde{\Delta},$ since any other number would lead to $\ell(e_i^\vee) \neq 1.$

The final claim, that each vertex of  $\tilde{\Delta}_i$ is a vertex of  $\tilde{\Delta},$ is an immediate
consequence of the fact that a point in  $\tilde{\Delta}_i$ cannot  be written as a convex combination 
of points in $\tilde{\Delta}$ if there is a non-zero term whose point $p$ has $e^\vee_i(p) = 0.$
\end{proof}
\begin{cor} {\bf (summand partition)}
\label{corollary:sum-partiton}
For an index $r$ reflexive Gorenstein cone, each $\tilde{\Delta}_i$ contains  exactly one $p_j$ in any expression
of 
$$\ell = p_1 + \cdots + p_r$$
as a sum of non-zero lattice points in $K.$
\end{cor}
\begin{prop} {\bf (dual splittings exist)}
If $K$ admits a complete splitting, then so does $K^\vee.$
\end{prop}
\begin{proof}
The proof is essentially a restatement of the middle of the proof of Proposition \ref{proposition:part-part}.
If we fix a splitting $E,$ the integral points of  $\tilde{\nabla}$ generate $K^\vee \cap \overline{N}$  
and are partitioned by the $ \tilde{\nabla}_i$'s.
So  $\ell^\vee \in \tilde{\nabla}_1 + \cdots +  \tilde{\nabla}_r$
and  any choice $\ell^\vee = e_1^\vee + \dotsm + e_r^\vee$ gives a complete splitting of $K^\vee.$
\end{proof}

\subsubsection{Dual splittings, complete intersections, and corresponding LG's}

\begin{defn} {\bf (the base toric variety $Y$)}
\label{definition:base-toric}
Given dual complete splittings $E$ and $E^\vee,$ define
$$
M = \{ m \in \overline{M} \ | \ e^\vee_j(n) = 0 \ \  \forall e^\vee_j \in E^\vee \}. 
$$
Within $M_\mathbf{R}$ we have the convex polyhedral set 
$$
\Delta_i = \{ m \in M_\mathbf{R} \ | \ 
m + e_i  \in \tilde{\Delta}_i
\}
$$
and the Minkowski sum
$$
\Delta = \sum_i \Delta_i.
$$
Denote by $\Sigma_Y \subseteq N = \operatorname{Hom}_\mathbf{Z}(M, \mathbf{Z})$ the normal fan of $\Delta,$ and write $Y$ for the 
toric variety of this fan.
\end{defn}

\begin{prop}{\bf (polar=convex) \cite[Proposition 3.18]{batyrev-nill}}
\label{prop:batyrev-nill}
The polar polytope $\Delta^*$ of $\Delta$ is the convex hull of the $\nabla_j$'s:
$$
\Delta^* = \operatorname{conv}(\{ \nabla_j \}_j).
$$
\end{prop}

\begin{defn} 
\label{definition:divisor-tau}
The proposition above and the integrality of the vertices of the $\nabla_j$'s allow us to define for each $e_i \in E$ a {\bf divisor}
$$
D_i =  \sum_{u_\rho \in \nabla_i}  D_\rho
$$ 
on $Y.$ 
\end{defn}

\begin{thm} {\bf (Cartier-ity)}
$D_i$ is Cartier with global section polytope $\Delta_i.$
\end{thm}
\begin{proof}
The theorem follows from establishing 
the polytope associated to 
$D_i$ is $\Delta_i.$  From this, $D_i$ is Cartier by the integrality of the vertices of $\Delta_i$ and 
Theorem 6.1.7 from Cox-Little-Schenck \cite{Cox-Little-Schenck}.

Recall that $m \in M_\mathbf{R}$ is in the polytope associated to 
$D_i$ if and only if 
$$D_i + \sum_{\rho \in \Sigma_Y(1)} u_{\rho}(m) D_\rho \geq 0.$$
Working $\rho$ by $\rho$ this is 
$$
u_\rho(m) + \left\{
\begin{array}{ll}
 1 & \text{if $u_\rho \in \nabla_i$}, \\
 0 &  \text{otherwise}\\
\end{array}
\right.
 \ \geq 0.
$$

Proposition \ref{prop:batyrev-nill} means the $u_\rho$'s equal the vertices of the convex hull of the $\nabla_i$'s 
and this provides enough to complete the proof.  Indeed, the transpose 
$$
\overline{N} \to N
$$
of the inclusion $M \hookrightarrow \overline{M}$ sends $\tilde{\nabla}$ to the convex hull of the $u_\rho$'s
So $u_\rho(m) = v(m)$ for an appropriate vertex of $\tilde{\nabla},$
and
$$
v(m + e_i) \geq 0
$$
becomes 
$$
u_\rho(m) \geq 0
$$
for those $u_\rho$ which do not come from $\nabla_i,$ and 
$$
u_\rho(m) + 1 \geq 0
$$
for the others.
\end{proof}

\begin{defn} {\bf ($\mathcal{V}, 
X, \Sigma$)}
With $\mathcal{V} = \bigoplus_i \mathcal{O}(D_i),$ we set
\begin{itemize}
\item $X =\operatorname{\underline{Spec}} \operatorname{Sym^\bullet} \mathcal{V},$ and
\item $\Sigma = \Sigma_X.$
\end{itemize}
\end{defn}

A complete description of the fan $\Sigma$ depends on knowing more information about the vertices  $\Delta.$  However, we can determine what we need  with what we know so far.

\begin{prop} {\bf ($N_X = \overline{N}, \Sigma(1)$ and $\Gamma(Y, \mathcal{V})$)}
\label{proposition:sigma-fan} 
The one-parameter subgroups of $X$ are naturally identified with $\overline{N}$.  Under this identification  
the primitive generators of the $1$-cones in  $\Sigma$ are exactly $\operatorname{vert}(\tilde{\nabla}) \cup \{e_i^\vee \}_i.$
Thus $|\mathbf{\Sigma}| = K^\vee,$ and 
$\Gamma(Y, \mathcal{V}),$ considered as functions on $X,$ have a basis made up of elements $\tilde{\Delta} \cap \overline{M}.$
\end{prop}
\begin{proof}
To produce $Y$ via the construction according to 
Definition \ref{definition:base-toric}, we begin with a pair of dual nef-partitions $E$ and $E^\vee.$
First there is a polytope $\Delta \subseteq M = (E^\vee)^\perp,$ 
defined in terms of $E.$
The fan of $Y$ is the inward normal fan of $\Delta$ and lives in the group $N$ which is $\mathbf{Z}$-dual to $M.$

Writing $F(-)$ for the free group, and following the 
 construction of the fan $\Sigma$ of a split bundle $\bigoplus_i \mathcal{O}(-D_i)$ over a toric variety of Corollary \ref{cor:split-bundle-fan} the fan lives in 
$$
\overline{N} = N \oplus F(\{\tau_i^\vee\}_i) \quad \quad ( = N \oplus \mathbf{Z}^{ \{ \tau_i \}_i })
$$
where $\tau_i^\vee$ is  $\mathbf{Z}$ dual to  the rational section $\tau_i$ of  $\mathcal{O}(D_i)$ with divisor $D_i.$
The $1$-cones are either the 
``vertical'' $1$-cones generated by the 
$$
\tau_i^\vee.
$$
or those the lifted from the $1$-cones of $\Sigma_Y.$
The generators of the lifts can be written abstractly as
$$
\widehat{u}_\rho = u_\rho + \sum_i \nu_{i,\rho} \tau^\vee_{i},
$$
where the coefficients come from the expansion
$$
D_i = \sum_\rho \nu_{i, \rho} D_\rho.
$$
For us this is 
$$
\widehat{u}_\rho = u_\rho + \tau^\vee_{i},
$$
for the unique $\nabla_i$ containing $u_\rho.$

Identification of these points with corresponding points in $K^\vee$
begins with identifying the ambient spaces. Writing $F(-)$ for the free group, the dual nef-partitions provide a
splitting
$$
\overline{M} =  M \oplus F(E).
$$
$M$ is already equal to the dual of $N,$ and we can further identify
the characters on the total space with $\overline{M}$ by the assignments
$$
e_i = \tau_i.
$$
To be sure, this is legitimate because the divisor of $\tau_i$ is supported on toric divisors
and it (up to constant rescaling) a character.

This identification immediately matches $\widehat{u}_\rho \in \nabla_i$ with 
a vertex of $\tilde{\nabla}_i,$  and the elements of $E^\vee$ corresponds with
the $\tau^\vee_i$'s.
\end{proof}

\begin{defn} {\bf ($\Gamma, \Gamma', \mathbf{W}, \mathbf{W}', \Sigma', Z, Z^\vee, \gamma \mapsto g, \gamma' \mapsto g'$)} In light of the theorem above, we set
\begin{itemize}
\item $\Gamma = \operatorname{Spec} \mathbf{Z}[\gamma^m]_{m \in \tilde{\Delta} \cap \overline{M}},$ 
\item $\mathbf{W} = \sum_{m \in \tilde{\Delta} \cap \overline{M}} \  \gamma_m x^m,$ and
\item $Z \subseteq Y \times \Gamma$ is the zero scheme of the universal section of  $\mathcal{V}.$
\end{itemize}
In addition we denote by $\Sigma',$ $\Gamma',$ $\mathbf{W}'$ and  $Z^\vee$ the analogous objects obtained by reversing the roles of $K$ and $K^\vee$ above.
Finally, the specializations are the identity maps $\Gamma \to \Gamma$
and $\Gamma' \to \Gamma'.$
\end{defn}

\begin{thm}{\bf (BB main)}  The fans $\Sigma$ and $\Sigma'$ are dual, and the toric Landau-Ginzburg models they define are obtained from the auxiliary Landau-Ginzburg models of $Z$ and $Z^\vee$ via base change along the inclusions 
$$
C(\Sigma') \hookrightarrow \Gamma \text{ \quad and  \quad } 
C(\Sigma) \hookrightarrow \Gamma'
$$
which set to $0$ any coefficient corresponding to a non-ray generating character.
\end{thm}
\begin{proof}
These statements are immediate from the observation that the primitive generators of the $1$-cones of $\Sigma'$ lie in $\tilde{\Delta}.$
\end{proof}

\subsection{Givental \cite{Givental-1998}}
In this case, one side of the mirror is a  complete intersection and the other is a Landau-Ginzburg model, so $\Gamma$ is the affine spaces of global sections, and $\mathbf{W}$ is the universal section considered as a function.  
On the other side, $\mathbf{W}'$ is 
formed by simply inserting variables for the coefficients of $W',$  and $\Gamma'$ 
is the affine space with these coefficient variables as coordinates. 

Most of the work is expressing $X'$ as a toric variety, and expanding $W'$ in characters.

\begin{defn}
{\bf (Givental's mirror)}
As input data, we have 
\begin{itemize}
\item a smooth projective toric variety $Y$ with $n$ torus invariant divisors,
\item $\ell$ nonnegative (i.e. basepoint free \cite[pg. 23]{Givental-1998}) line bundles $(\mathcal{L}_a)_a$ such that $\omega_a^\vee \otimes \bigotimes_a \mathcal{L}^\vee_a$ is nonnegative, and
\item a choice of integral symplectic basis $(p_1, \dotsc, p_k)$ of $H^2(Y, \mathbf{Z}).$ 
\end{itemize}
From this we have 
 integers $\ell_{ia}$ defined by 
$$
\operatorname{ch}(\mathcal{L}_a) = \sum_i \ell_{i a} p_i
$$
and integers $m_{i \rho}$ defined by 
$$
\operatorname{ch}(\mathcal{O}(D_\rho)) = \sum_i m_{i \rho} p_i
$$
for the toric divisor $D_\rho$ corresponding to $\rho \in \Sigma_Y(1).$

{\bf Givental's mirror Landau-Ginzburg model} is given by the scheme
$$
E \subseteq    \mathbf{C}^{n}_w \times  \mathbf{C}^{\ell}_v \times (\mathbf{C}^\times)^k_q 
$$
which is the closure of the zero locus in 
$(\mathbf{C}^\times)^{n}_u \times  (\mathbf{C}^\times)^{\ell}_v \times (\mathbf{C}^\times)^k_q $ 
of the equations
\begin{equation}
\label{equation:Eq-equations}
 \prod_\rho w_\rho^{m_{i \rho}} = q_i \prod_a v_a^{\ell_{i a}} \quad  i = 1, \dotsc, k
\end{equation}
equipped with the function
\begin{equation}
\label{equation:givental-W'}
W' \colon E \to \mathbf{C}
\end{equation}
given by $W' =  \sum_a v_a -\sum_\rho w_\rho.$
\end{defn}

\begin{defn} {\bf ($\mathcal{V}, 
X, N, \Sigma, \Gamma, \mathbf{W}, \gamma \mapsto g$)}
With $\mathcal{V} = \bigoplus_a \mathcal{L}_a,$ we set
\begin{itemize}
\item $X =\operatorname{\underline{Spec}} \operatorname{Sym^\bullet} \mathcal{V},$ 
\item $\overline{N} = N_Y \oplus \mathbf{Z}^{\{ e_a\}_a},$ 
\item $\Sigma = \Sigma_X \subset \overline{N},$ 
\item $\Gamma = \operatorname{Spec} \mathbf{Z}[\gamma_\frak{m}]_{\frak{m} \in \Xi}$ where 
$
\Xi = \{ 
\mu + e_a \ | \ \mu \in \Delta_a \cap M_Y
\}$ for the global section polytopes $\Delta_a$ of the $\mathcal{L}_a$'s,
\item $\mathbf{W} = \sum_{\frak{m} \in \Xi} \gamma_{\frak{m}} x^\frak{m},$ and
\item the specialization $\operatorname{Spec} \mathbf{C} \to \Gamma$ depends on the equations of 
the complete intersection chosen.
\end{itemize}
\end{defn}

\begin{lem} {\bf ($F$-split)}
\label{lemma:split}
Writing $F(-)$ for the free group,
$H^2(Y, \mathbf{Z})$ is the cokernel of the map from characters to divisors on $Y$
$$
0 \to M_Y \to F(\{D_{\rho} \ | \ \rho \in \Sigma_Y(1) \})
\to H^2(Y, \mathbf{Z}) \to 0.
$$
Furthermore, one can find $k$ divisors $D_{\rho_i}$
such $\bigoplus_i \mathbf{Z} \, D_{\rho_i}$ maps isomorphically to  $H^2(Y, \mathbf{Z}).$
\end{lem}
\begin{proof}
Both statements follow from the facts that $Y$ is smooth and complete.
For the first, see, for instance, Theorems 4.1.3, 4.2.1, and 12.3.2 of Cox-Little-Schenck \cite{Cox-Little-Schenck}.
For the second, since $Y$ is smooth one can choose a torus fixed point, and
the $k$ divisors
which do not pass through it map to a basis for $\operatorname{H}^2(Y, \mathbf{Z}).$
\end{proof}

\begin{lem} {\bf (torus of $E$)}
\label{lemma:long-one}
Since there is a natural injection $\Sigma_Y(1) \to \Sigma_X(1),$
we can combine the inclusion $\iota \colon M_X \to F(\{D_{\rho} \ | \ \rho \in \Sigma_X(1) \})$ and the graph of the map 
$$s \colon H^2(Y, \mathbf{Z}) \to  F(\{D_{\rho} \ | \ \rho \in \Sigma_Y(1) \})$$ formed from an isomorphism 
$
H^2(Y, \mathbf{Z}) \to  \bigoplus_i \mathbf{Z} \, D_{\rho_i}
$
to yield an inclusion 
$$
\left[
\begin{array}{cc}
\iota  & s \\
0 & \mathbf{1}
\end{array}
\right]  \colon 
M_X \oplus H^2(Y, \mathbf{Z})  \hookrightarrow F(\{D_{\rho} \ | \ \rho \in \Sigma_Y(1) \}) \oplus H^2(Y, \mathbf{Z}).
$$
Applying the functor
 $\operatorname{Spec} \mathbf{C}[\operatorname{Hom}_\mathbf{Z}(-, \mathbf{Z})]$
 identifies 
 \begin{itemize}
 \item $\operatorname{Spec} \mathbf{C}[\operatorname{Hom}_{\mathbf{Z}}(H^2(Y, \mathbf{Z}), \mathbf{Z})]$ with $(\mathbf{C}^\times)^k_q,$ 
 \item  $\operatorname{Spec}(\mathbf{C}[N_X]) \times (\mathbf{C}^\times)^k_q$ with $E \cap (\mathbf{C}^\times)^{n}_u \times  (\mathbf{C}^\times)^{\ell}_v \times (\mathbf{C}^\times)^k_q$, and 
\item the inclusion with
 $$
 E \cap (\mathbf{C}^\times)^{n}_u \times  (\mathbf{C}^\times)^{\ell}_v \times (\mathbf{C}^\times)^k_q \ \ 
\hookrightarrow \ \ (\mathbf{C}^\times)^{n}_u \times  (\mathbf{C}^\times)^{\ell}_v \times (\mathbf{C}^\times)^k_q.
 $$
\end{itemize}
\end{lem}
\begin{proof}
Note $H^2(X, \mathbf{Z}) = H^2(Y, \mathbf{Z}).$  
If  $H^2(X, \mathbf{Z})$ is identified with $\mathbf{Z}^k$  with the $(p_i)_i$ basis, then matrix 
$$[(m_{i \rho})_{i \rho} | -(\ell_{i a})_{i a}]$$
is the map from the group of torus invariant divisors if $X$ to $H^2(X, \mathbf{Z}).$  
To verify this, we check that it is the kernel of the 
 transpose $U^t \colon \mathbf{Z}^n \oplus \mathbf{Z}^k \to N_X$
 of the first non-zero map in the exact sequence
$$
0 \to M_X \to \mathbf{Z}^n \oplus \mathbf{Z}^\ell \to \mathbf{Z}^k \to 0
$$
The map $U^t$ sends
the standard bases vectors to the primitive generators of rays in $\Sigma_X(1).$  
Following Corollary \ref{cor:split-bundle-fan}, we have  explicitly 
$$
U = \left[
\begin{array}{cc}
(u_\rho)_{\rho} & (\alpha_{\rho a})_{\rho a} \\
0 & I
\end{array}
\right].
$$
where 
we write $(u_\rho)_{\rho}$ for the matrix with rows $u_\rho$ for $\rho \in \Sigma_Y(1)$
and  $(\alpha_{\rho a})_{\rho a}$ is obtained from the expansions 
$$
D_a = \sum_\rho \alpha_{ \rho a } D_\rho
$$
of 
 toric divisors $D_a$ such that  $\mathcal{L}_a = \mathcal{O}_Y(D_a).$
 The $\alpha$'s satisfy
 \begin{equation}
\label{equation:L-equation}
\ell_{ia} = \sum_\rho m_{i \rho}\alpha_{\rho a}.
\end{equation}
Finally, $[(m_{i \rho})_{i \rho} | -(\ell_{i a})_{i }] U = 0$ by virtue of equation (\ref{equation:L-equation}) and  relations for the rays of $Y$ 
$$\sum_{\rho}m_{i \rho} u_{\rho}= 0.$$

The identification of 
$
\operatorname{Spec} \mathbf{C}[N_X] \times (\mathbf{C}^\times)^k_q
$ with
$E \cap (\mathbf{C}^\times)^{n}_u \times  (\mathbf{C}^\times)^{\ell}_v \times (\mathbf{C}^\times)^k_q$
comes from plugging 
the exact sequence
$$
0 \to M_X \oplus H^2(X, \mathbf{Z})  \to F(\{D_{\rho} \ | \ \rho \in \Sigma_Y(1) \}) \oplus H^2(X, \mathbf{Z}) \to H^2(X, \mathbf{Z}) \to 0
$$
into 
the functor $\operatorname{Spec} \mathbf{C}[\operatorname{Hom}_\mathbf{Z}(-, \mathbf{Z})].$ The first non-zero map is 
\begin{equation}
\label{equation:box-matrix}
\left[
\begin{array}{cc}
\iota  & s \\
0 & \mathbf{1}
\end{array}
\right]
\end{equation}
and the second is $[(m_{i \rho})_{i \rho} | -(\ell_{i a})_{i } | -\mathbf{1}].$ This yields an exact sequence of groups
$$
1 \to \operatorname{Spec} \mathbf{C}[N_X] \times (\mathbf{C}^\times)^k_q \to 
(\mathbf{C}^\times)^{n}_u \times  (\mathbf{C}^\times)^{\ell}_v \times (\mathbf{C}^\times)^k_q 
\to 
(\mathbf{C}^\times)^k_{q'} \to 1
$$
where the map to $(\mathbf{C}^\times)^k_{q'}$ is given by 
$$
q'_i = 
( \prod_\rho w_\rho^{m_{i \rho}}) /( q_i \prod_a v_a^{\ell_{i a}}).
$$
\end{proof}

\begin{prop} {\bf ($\Sigma_{X'}$)}
Denote $H = \operatorname{Spec} \mathbf{C}[\operatorname{Hom}_\mathbf{Z}(H^2(Y, \mathbf{Z}), \mathbf{Z})].$
 Then $M_X= M_Y \oplus F(\{ e_a \}_a),$ and
$$E=\operatorname{Spec} \mathbf{C}[K^\vee \cap N_X] \times H$$
where $K^\vee$ is the cone dual to the global function cone of $X:$
$$
 K = \{ t_1 (\mu_1 + e_1) + \cdots + t_\ell (\mu_\ell + e_\ell) \in (M_X)_\mathbf{Q} \ | \ t_a \geq 0 \text{ and } \mu_a \in \Delta_a \text{ for all }a \}.
$$
\end{prop}
\begin{proof}
Functions which are global on $E$ are exactly those on 
$E \cap (\mathbf{C}^\times)^{n}_u \times  (\mathbf{C}^\times)^{\ell}_v \times (\mathbf{C}^\times)^k_q $ 
which are restrictions from
$\mathbf{C}^{n}_u \times  \mathbf{C}^{\ell}_v \times (\mathbf{C}^\times)^k_q.$
For characters, these are elements of  
 $$
 N_X 
 \oplus 
\mathbf{Z}^k
  $$
  which come from $\mathbf{N}^n \oplus \mathbf{N}^\ell \oplus \mathbf{Z}^k$
  under the transpose of 
  $$
  \left[
\begin{array}{cc}
\iota  & s \\
0 & \mathbf{1}
\end{array}
\right]
  $$ from
  (\ref{equation:box-matrix}).

An element $(\frak{n}, z) \in N_X \oplus \mathbf{Z}^k$ defines a global characters 
 if and only if both $(\frak{n}, 0)$ and $(0, z)$ are global characters. 
   The $\mathbf{1}$ in the lower righthand corner guarantees $0 \oplus \mathbf{Z}^k$ is made up of global characters. So the question reduces to knowing for which elements $\frak{n}$ of $N_X$
   the character $(\frak{n}, 0)$ is global.

The $\frak{n}$'s we are interested in are exactly those which are global on $X',$
and this translates into the condition $\frak{n}$ is in the cone dual the global functions on $X.$ 
An element $n \in N_X$ comes from $\mathbf{N}^n \oplus \mathbf{N}^\ell$
  if and only if $\frak{n}$ is non-negative on any element of $M_X$ which maps into 
  $\mathbf{N}^n \oplus \mathbf{N}^\ell.$  These are exactly the elements of 
   $M_X$ which define a global function on $X.$
  \end{proof}

\begin{cor} {\bf ($\Sigma_{X'}(1) \subset \Gamma(Y, \mathcal{V})$)}
Take $X' = \operatorname{Spec} \mathbf{C}[K^\vee \cap N_X].$  Then
$\Sigma_{X'}$ is the cone $K$ and all its facets.  
Furthermore, the primitive integral generators of elements of $\Sigma_{X'}(1)$
are all of the form $v + e_a$ for $v \in \Delta_a.$
\end{cor}
\begin{proof}
The only statement that isn't immediate is the one about the generators.  However, 
we know that $\Delta_a$ has integral vertices, so this guarantees result.
\end{proof}

\begin{prop} {\bf ($W'$ on $X' \times H$)}
Enumerate $\Sigma_X(1)$ so that the first $k$ rays correspond to the divisors in Lemma \ref{lemma:split}.
Then as a function on $E = X' \times H$ we have  
\begin{equation}
\label{equation:w'-specialized}
W' = \sum_a x'_a - \sum_{j=1}^k \prod_{i} q_i^{-s_{j i}} x'_{\rho_j} - \sum_{j=k+1}^n x'_{\rho_j}
\end{equation}
where 
\begin{itemize}
\item $x_a' = v_a,$  
\item $x'_{\rho_{j}} =   \prod_i q_i^{s_{j i}}  w_{\rho_{j}} $ for $j \leq k,$ and
\item $x'_{\rho_j} = w_{\rho_j}$ for $j > k.$
\end{itemize}
are characters on $X'$ via
$\mathbf{Z}^n \oplus \mathbf{Z}^\ell \to N_X.$ The integers $s_{ji}$ 
come from the map $s$ of Lemma \ref{lemma:long-one} in the $(p_i)_i$ basis.  Furthermore, 
the characters $x'_a$ and $x'_\rho$ are exactly the primitive generators 
for the rays in $\Sigma_X(1).$
\end{prop}
\begin{proof}
The way functions pull back under the inclusion $X' \times H \hookrightarrow \mathbf{C}^n \times \mathbf{C}^\ell (\mathbf{C}^\times)^k$ is described by the transpose of the map in (\ref{equation:box-matrix}).  Both 
the characters $x'_a$ and $x'_\rho$ and the primitive generators 
for the rays in $\Sigma_X(1)$ are the images of the standard basis vectors under 
$\mathbf{Z}^n \oplus \mathbf{Z}^\ell \to N_X.$
\end{proof}

\begin{defn} {\bf ($\Gamma', \mathbf{W}', \gamma' \mapsto g'$)}
The expansion of $W'$ in (\ref{equation:w'-specialized}) and the identification of 
its terms with the primitive generators 
for the rays in $\Sigma_X(1)$ gives us the definitions
\begin{itemize}
\item $\Xi' = \{ \frak{n} \in N_X \ | \ \frak{n}\text{ is a primitive generator of } \Sigma_X(1)\},$
\item $\Gamma' = \mathbf{Z}[\gamma'_\frak{n}]_{\frak{n} \in \Xi'},$ 
\item $\mathbf{W}' = \sum_{n \in \Xi'} \gamma'_{\frak{n}} {x'}^\frak{n},$ and
\item the specialization $H \to \Gamma'$ sets the coefficients of $\mathbf{W}'$ to match those in  (\ref{equation:w'-specialized}).
\end{itemize}
\end{defn}

\begin{thm} {\bf (Givental's mirror main)}
There are inclusions as coordinate subspaces
$$
C(\Sigma') \to \Gamma \text{ \quad and \quad } C(\Sigma) \to \Gamma'
$$
which produce the dual pair of toric Landau-Ginzburg models formed from $\Sigma$ and $\Sigma'$ via base change.
Furthermore,   $C(\Sigma) \to \Gamma'$ is an isomorphism.
\end{thm}
\begin{proof}
Omitted.
\end{proof}

\begin{rem}
{\bf (independence of choice of $p$'s)}
Notice that one consequence of this theorem is that Givental's mirror is 
independent of the choice of basis $(p_i)_i$ of $H^2(Y, \mathbf{Z}).$  Indeed,
one need not even assume the existence.
\end{rem}

\subsection{Hori-Vafa \cite{Hori-Vafa}}
\label{section:hori-vafa}
Hori-Vafa gives an expression \cite[Equation (7.78)]{Hori-Vafa}
for the BPS mass of a certain D-brane
related to complete intersection $Z$ in a toric variety $Y$.  Their expression is
the integral (\ref{equation:bps-mass}).This integral
is the pullback of and integral on a algebraic torus. 
This torus the function $W'$ appearing in the integrand match Givental's Landau-Ginzburg mirror.

\begin{defn}{\bf (Hori-Vafa BPS mass \cite[Equation (7.78)]{Hori-Vafa})}
\label{definition:bps-mass}
Given a toric variety $Y,$ $\ell$ global sections $G_\beta$
of line bundles $\mathcal{L}_\beta,$ and  a relative homology class, $[\gamma] \in \operatorname{H_n}(\mathbf{C}^N \times \mathbf{C}^\ell; B)$
 where $B = W'^{-1}(r)$ is
a fiber of 
$$W' = \sum_{i = 1}^N \exp(-Y_i) + \sum_{\beta = 1}^\ell \exp(-Y_{P_\beta})$$
 over a sufficiently large  $r \in \mathbf{R},$
the  BPS mass for a $D$-brane wrapping  $k$ is
\begin{equation}
\label{equation:bps-mass}
\begin{array}{rrll}
\Pi_\gamma  & = &  
\int_\gamma &   \\
& & & \prod_{i=1}^N dY_i  \\
& & & \prod_{\beta = 1}^\ell \exp(-Y_{P_{\beta}}) dY_{P_\beta} \\
& & & \prod_{a = 1}^k \delta(\sum_{i = 1}^N Q_{i a} Y_i - \sum_\beta d_{\beta a} Y_{P_{\beta}} - t_a) \\
& & &  \exp(-W')
\end{array}
\end{equation}
where 
\begin{itemize}
\item $\Sigma_Y(1) = \{ \rho_1, \dotsc, \rho_N \}, $ 
\item $\{ \eta_1, \dotsc, \eta_k\}$ is an integral basis of $H^2(\mathbf{Y}, \mathbf{Q}),$ and 
\item the class of  $\mathcal{L}_\beta$ is $ \sum_a d_{\beta a} \eta_a.$
\end{itemize}
\end{defn}

The delta functions in the integrand of Equation (\ref{equation:bps-mass})
are interpreted to mean a restriction of the integral to the subset on which 
the arguments vanish.
\begin{defn}{\bf ($\delta$ functions and forms)}
\label{definition:delta-integrand}
If $\phi$ is a differential form, the expression
$$
\int_\gamma \delta(g_1) \dotsm \delta(g_k) \phi
$$
means
$$
\int_{\gamma \cap \{ g_1= \dotsm = g_k = 0 \}} \psi
$$
where 
$$
\phi = dg_1 \wedge \dotsm  \wedge dg_k \wedge \psi.
$$
The apparent ambiguity in this definition is addressed in Lemma \ref{lemma:delta_forms_unambiguous}.
\end{defn}
\begin{lem} {\bf (unambiguity of $\delta$ valued forms)}
\label{lemma:delta_forms_unambiguous}
Given functions $f_1, \cdots, f_k,$ write 
$$
\mathbf{V}(g) = \{ p \ | \ f_1(p) = \dotsm = f_k(p) = 0 \}
$$
and 
$$
\mathbf{U}(dg) =  \{ p \ | \ df_1|_p,  \dotsc,  df_k|_p \text{ are linearly independent} \}.
$$
Any forms $\psi$ and $\psi'$ such that 
$$df_1 \wedge \dotsm  \wedge df_k \wedge \psi = df_1 \wedge \dotsm  \wedge df_k \wedge \psi',$$
satisfy
$$
\int_\gamma \psi = \int_\gamma \psi'
$$
for any simpex  $\gamma \colon \Delta^n \to \mathbf{V}(g)$ such that 
$\gamma^{-1}(\mathbf{U}(dg))$
is dense in $\Delta^n.$ 
\end{lem}
\begin{proof}
Consider a point $p \in \mathbf{U}(dg)\cap \mathbf{V}(g).$  About this point we can find 
coordinates $g \colon U \to \mathbf{R}^n$ such that the first $k$ coordinates are 
$f_1, \dotsc, f_k.$  On this neighborhood, 
$$
\psi = \psi' + df_1 \wedge \psi_1 + \dotsm +  df_k \wedge \psi_k
$$
for forms $ \psi_1, \dotsc, \psi_k.$  However, any $df_i$ for $i = 1, \dotsc, k$ annihilates any tangent vector in 
$\mathbf{V}(g).$ So provided the image of $\gamma$ is in $\mathbf{V}(g)$ we know 
$$
\int_{\gamma^{-1}(\mathbf{U}(dg))} \gamma^*\psi = \int_{\gamma^{-1}(\mathbf{U}(dg))} \gamma^*\psi'.
$$
Finally, $\gamma^{-1}(\mathbf{U}(dg))$ is dense in $\Delta^n,$ so its complement has measure zero 
and its omission doesn't affect the integral.
\end{proof}

\begin{prop}{\bf (variables on the torus)}
For $
W' = -\sum_{i = 1}^N w_i -  \sum_{\beta = 1}^\ell v_\beta,
$
the integrand of Equation (\ref{equation:bps-mass})
is the pullback 
of 
\begin{equation}
\label{equation:form-on-torus}
(-1)^{N+\ell} \exp(-W') \ \ \prod_a \delta( \prod_i w_i^{Q_{ia}} - q_a \prod_\beta v_\beta^{d_{\beta a}} ) 
\ \   d\ln(w_1) \wedge \dotsm \wedge d\ln(w_N) \wedge dv_1 \wedge \dotsm \wedge dv_\ell
\end{equation}
along 
$$
\mathbf{C}_t^k \times \mathbf{C}^N_{Y} \times \mathbf{C}^\ell_{Y_P} \to 
(\mathbf{C}^\times)_q^k \times (\mathbf{C}^\times )^N_w \times (\mathbf{C}^\times)^\ell_v
$$
given by 
\begin{itemize}
\item $w_i = \exp(-Y_i),$  
\item $v_\beta = \exp(-Y_{P_\beta}),$ and
\item  $q_a = \exp(-t_a).$
\end{itemize}
\end{prop}
\begin{proof}
Direct substitution.
\end{proof}

\begin{defn}
{\bf (Hori-Vafa mirror LG)}
Denote by $Z$
 the subsheme of $Y$ defined by  $G_\beta = 0$ for $\beta = 1, \dotsc, \ell.$
The Hori-Vafa mirror Landau-Ginzburg mirror to $Z$ is
 $(X',W')$ where $X'  \subseteq (\mathbf{C}^\times)_q^k \times \mathbf{C}^N_w \times \mathbf{C}^\ell_v$
is given by the $k$ equations 
 $$
 \prod_i w_i^{Q_{ia}} - q_a \prod_\beta v_\beta^{d_{\beta a}} 
 $$
 and 
$
W' = -\sum_{i = 1}^N w_i -  \sum_{\beta = 1}^\ell v_\beta.
$ 
\end{defn}

\begin{thm} {\bf(computation of BPS mass on the dual)}
Using the interpretation of Definition \ref{definition:delta-integrand}, 
the BPS mass given in Equation (\ref{equation:bps-mass}) equals 
$$
\int_{\overline{\gamma}} \psi
$$
for the form in (\ref{equation:form-on-torus}) and the class $\overline{\gamma} \subseteq X'$
that is the image in $X'$ of $\gamma \cap \{g_1 = \dotsc = g_k = 0\},$ 
where the $g$'s are the arguments of the $\delta$ functions.
\end{thm}

\begin{rem}{\bf (comparison to Givental's mirror)}
Notice  that $X'$ and Givental's $E_q$ in Equation (\ref{equation:Eq-equations}) are the same.  
However,  Givental's $W'$ in Equation (\ref{equation:givental-W'}) has a different sign
in front of the $w$-variables than the Hori-Vafa $W'.$  This difference means that both 
are obtained from the same auxiliary Landau-Ginzburg model via
different base change morphisms. 
\end{rem}

\subsection*{Acknowledgements}   
The author would like to thank E. Sharpe for his interest and for first pointing out Berglund-H\"ubsch mirror symmetry, and M. Krawitz for explaining his dual group construction.   Correspondence and conversations with L. Borisov, D. Cox  and B. Nill on several aspects of toric geometry have been invaluable.  Finally, we thank L. Katzarkov, E. Gasparim and Y. Ruan for their great encouragement and patience.

\newpage
\bibliography{toric_potential}
\bibliographystyle{halpha}  

\noindent
{Department of Mathematics, Drexel University, Philadelphia, PA 19104\\
\texttt{pclarke@math.drexel.edu}}

\end{document}